\documentclass[12 pt]{article}%
\usepackage{amsmath, amsfonts, amsthm, color,latexsym}
\usepackage{amsmath}
\usepackage{amsfonts}
\usepackage{amssymb}
\usepackage{color, soul}
\usepackage{graphicx}%
\usepackage{enumerate}
\usepackage[utf8]{inputenc}
\usepackage[T1]{fontenc}
\allowdisplaybreaks[4]
\providecommand{\U}[1]{\protect\rule{.1in}{.1in}}
\newtheorem{theorem}{Theorem}[section]
\newtheorem{proposition}[theorem]{Proposition}
\newtheorem{corollary}[theorem]{Corollary}
\newtheorem{example}[theorem]{Example}

\newtheorem{remark}[theorem]{Remark}

\newtheorem{lemma}[theorem]{Lemma}
\newtheorem{final remark}[theorem]{Final Remark}
\newtheorem{definition}[theorem]{Definition}
\begin{document}
	
	\title{\sc Demi Weak Dunford-Pettis Operators} 
\author{ Joilson Ribeiro\thanks{joilsonor@ufba.br} ~ Fabr\'icio Santos\thanks{fabricio.antonio@unir.br\thinspace \hfill\newline\indent2010 Mathematics Subject
		Classification: 46B42, 46B40, 46A40.\newline\indent Key words: Banach  spaces, Demi Dunford-Pettis, Banach Lattices, Weakly Demi Dunford Pettis.}}
\date{}
\maketitle

\begin{abstract} 
	
	In this paper, our main goal is to introduce the class of Demi Weak Dunford-Pettis operators. We also investigate its relationship with the classes of weak Dunford–Pettis and demi Dunford–Pettis operators, as well as conditions under which this class coincide with the classes of demi Dunford–Pettis and demicompact operators. In the final section, we study some properties of this new class in the setting of Banach lattices.

\end{abstract}

\section{Introduction and background}

The Dunford-Pettis operators, whose concept can be found in \cite{AB85}, play a central role in the theory of measure spaces and also in Banach spaces in general, as we can see in \cite{DU77}. Recall this concept: Let X and Y be two Banach spaces. An operator $T\colon X \rightarrow Y$ is called a Dunford-Pettis
operator if T carries weakly convergent sequences onto norm convergent sequences. In other words, if $x_j \rightarrow 0$ as $j \rightarrow \infty$ in $\sigma(X; X')$, then $\|T(x_j)\|_Y \rightarrow 0$, as $j \rightarrow \infty$. Alternatively: A (bounded linear) operator $T\colon X \rightarrow Y$ is a Dunford-Pettis operator if and only if T carries relatively weakly compact sets onto norm totally bounded sets. A Banach space $X$ is said to have the Schur property if $Id_{X}\colon X \rightarrow X$, given by $Id_{X}(x) = x$, is Dunford-Pettis. The Banach space $\ell_1$ has the Schur property (\cite[Theorem $2.5.18$]{MN91}).

The concept of Demi Dunford-Pettis operation was introduced by Benkhaled, Hajji and Jeribi in \cite{BHJ22} as a class that generalizes the concept of Dunford-Pettis operators. This class has been studied by several authors, such as ({\cite{BKA25}, \cite{BKAB24}, \cite{GU22}}). Recall this definition: Let $X$ a Banach space, an operator $T\colon X \rightarrow X$ is said to be Demi Dunford Pettis if, for every $(x_j)_{j=1}^{\infty}$ in $X$ such that $x_j \rightarrow 0$ in $\sigma(X; X')$ with $\|x_j - T(x_j)\| \rightarrow 0$ then $\|x_j\| \rightarrow 0$. It is easy to see that every Dunford-Pettis operator is a Demi-Dunford-Pettis operator, but the converse is generally not true, by example, let the operator $Id_{\ell_{\infty}} \colon \ell_{\infty} \rightarrow \ell_{\infty}$, given by, $Id_{\ell_{\infty}}((x_j)_{j=1}^{\infty}) = (x_j)_{j=1}^{\infty}$ called identity, the operator $-Id_{\ell_{\infty}}\colon \ell_{\infty} \rightarrow \ell_{\infty}$ is Demi Dunford-Pettis but it is not a Dunford-Pettis operator.

In this paper, inspired by the ideas of generalizations contained in the papers \cite{BEJ20,  BHJ22, BHJ22(2), KO19, P66}, we present a new class that generalizes the concept of weak Dunford-Pettis operators; we will call them Demi Weak Dunford-Pettis operators. An operator $T$ on a Banach space $X$ is said Demi Weak Dunford-Pettis (DWDP in short), if every sequences $(x_j)_{j=1}^{\infty}$ and $(f_j)_{j=1}^{\infty}$ in $X$ and $X'$, respectively, weakly convergente, that is, $x_j \rightarrow 0$ in $\sigma(X, X')$ and $f_j \rightarrow 0$ in $\sigma(X', X'')$ and $\|x_j - T(x_j)\| \rightarrow 0$ we have $|f_j(x_j)| \rightarrow 0$, as $j \rightarrow \infty$. Throughout this work, we study this class and its relationships with the classes of weak Dunford-Pettis (see Proposition \ref{DwDP1}), Demi Dunford-Pettis (see Proposition \ref{DDPWDDP}), and Demicompact operators (see \cite{P66}). We also present concepts equivalent to the main definition (see Proposition \ref{WDDPequi}), in order to establish various ways to verify when an operator on Banach spaces is weakly Demi-Dunford-Pettis. In the last section, we presented two results on the class of weakly Demi-Dunford-Pettis operators on Banach lattices related to the domination property of this class (see Theorem \ref{Desig} and Corollary \ref{Desig2}).

Throughout this text, $X$ will represent either a Banach space or a Banach lattice. In the especific moment, will be point out which of the two concepts will be used in the context. We will denote by $X'$ and $X''$ the topological dual and bidual of $X$. We also denote the positive cone of $X$ by $X^{+}$ Let $(x_j)_{j=1}^{\infty}$ in $X$, we indicate by $x_j \overset{w}{\rightarrow} x$ to mean that $x_j \rightarrow x$ in $\sigma(X, X')$, as $j \rightarrow \infty$.

To show our results we need to fix some notation and recall some definitions that will be used in this paper. A Banach lattice is a Banach space $(E, \|\cdot\|)$ such that $E$ is a vector lattice and its norm satisfies the following property: for each $x, y \in E$ such that $|x| \le |y|$, we have $\|x\| \le \|y\|$. The topological dual of a Banach lattice is also a Banach lattice, in other words, if $X$ is a Banach lattice, then $X'$ will also be a Banach lattice, endowed with the dual norm. We say that $X$ has the Dunford-Pettis property if $|f_j(x_j)| \rightarrow 0$, as $j \rightarrow \infty$ for every weakly null pair of sequences $((x_j)_{j=1}^{\infty}, (f_j)_{j=1}^{\infty})$ in $X \times X'$.

We will use the term operator $T\colon X \longrightarrow Y$ betwen two Banach spaces to mean a bounded linear mapping. An operator $T\colon X \longrightarrow Y$ betwen two Banach Lattices is said to be positive if $T(x) \ge 0$ in $Y$ whenever $x \ge 0$ in $X$. We write $S \le T$ if $(T - S)x \ge 0$ for every $x \in X^+$. We say that $S$ is dominated by $T$. It is well known that each positive linear mapping on Banach Lattice is continuous. For more details about Banach Lattices and positive opertors, we indicate the references \cite{AB85} and \cite{MN91}.


\section{Demi Weak Dunford-Pettis Operators on Banach Spaces}

In this section, we introduce the concept of a Demi Weak Dunford–Pettis linear operator and investigate its relationship with existing classes of operators in the literature, namely, weak Dunford-Pettis linear operators (see \cite{AB85}) and Demi Dunford-Pettis linear operators (see \cite{BHJ22}).

\begin{definition}\rm
	Let $X$ be a Banach space. An operator $T\colon X \longrightarrow X$ is said to be Demi Weak Dunford-Pettis (DWDP in short), if for every sequence $(x_j)_{j=1}^{\infty}$ in $X$ and $(f_j)_{j=1}^{\infty}$ in $X'$ such that $x_j \overset{w}{\longrightarrow} 0$, $f_j \overset{w}{\longrightarrow} 0$ and $$\|x_j -  T(x_j)\| \longrightarrow 0$$ as $j\rightarrow \infty$, we have $|f_j(x_j)| \rightarrow 0$ as $j \rightarrow \infty$.
\end{definition}

The following result provides an equivalent characterization of weak Dunford–Pettis operators.

\begin{proposition}\label{WDDPequi}
	Let $X$ be a Banach space. An operator $T\colon X \longrightarrow X$ is weak demi Dunford–Pettis if, and only if, for every pair of sequences $(x_j)_{j=1}^{\infty} \subset X$ and $(f_j)_{j=1}^{\infty}\subset X'$ satisfying $x_j \overset{w}{\rightarrow} x$, $f_j \overset{w}{\rightarrow} f$ and $\left\|(x_j - T(x_j)) - (x - T(x))\right\| \rightarrow 0$ we have
	\begin{equation*}
		|f_j(x_j) - f(x)| \rightarrow 0
	\end{equation*}
as $j \to \infty$.
\end{proposition}



The following result shows that every Weak Dunford–Pettis linear operator is Demi Weak Dunford–Pettis operator.

\begin{proposition}\label{DwDP1}
Let $X$ be a Banach space. Every weak Dunford-Pettis operator $T\colon X \rightarrow X$ is a demi weak Dunford-Pettis operator.
\end{proposition}
\begin{proof}
Let $(x_j)_{j=1}^{\infty}$ be a sequence in $X$ and let $(f_j)_{j=1}^{\infty}$ be a sequence in $X'$ such that $x_j \overset{w}{\rightarrow} 0$, $f_j \overset{w}{\rightarrow} 0$ and $\|x_j-T(x_j)\| \rightarrow 0$. For every $j\in\mathbb{N}$, we have
	\begin{equation*}
		|f_j(x_j)| \le \|f_j(x_j-T(x_j))\| + \|f_j(T(x_j))\| \le \|f_j\| \left(\|(x_j-T(x_j)\|\right) + \|f_j(T(x_j))\|.
	\end{equation*}
	Since $(f_j)_{j=1}^{\infty}$ is weakly convergent in $X'$, it is norm bounded, and  $T$ is a weak Dunford-Pettis then operator
	\begin{equation*}
		|f_j(x_j)| \rightarrow 0
	\end{equation*}
	as $j\rightarrow \infty$. 
\end{proof}

It is important to note that the opposite path of proposition \ref{DwDP1} does not always happen, as we can see in the next example.

\begin{example}
	Consider the example
	\begin{equation*}
		Id_{\ell_2}\colon \ell_2 \rightarrow \ell_2
	\end{equation*}
	given by, $Id_{\ell_2}\left((x_j)_{j=1}^{\infty}\right) = \left(x_j\right)_{j=1}^{\infty}$. It's not difficult to see that $-Id_{\ell_2}$ is Weakly Demi Dunford-Pettis operator but when consider the sequence $(e_j)_{j=1}^{\infty}$ in $\ell_2$,  where $e_j$ is the sequence with the $j$th entry equals to $1$ and others are zero and $(f_j)_{j=1}^{\infty}$ on $\ell_2'$. It is canonical that $e_j \overset{w}{\longrightarrow} 0$ on $\ell_2$, and each $f_j$ can be associated to $e_j$. So,
	\begin{equation*}
		|f_j(-Id_{\ell_{\infty}}(e_j))| = |f_j(e_j)| = 1 \not\to 0.
	\end{equation*}
Then, $Id_{\ell_2}$ is not Weakly Dunford-Pettis.
	
\end{example}

The following result provides a necessary and sufficient condition for an operator to be weak demi-Dunford–Pettis.

\begin{theorem}\label{WDDP1}
	Let $X$ be a Banach space. The following statement are equivalent
	\begin{enumerate}
		\item[$(a)$] Every continuous operator $T \colon X \longrightarrow X$ is Weak Dunford-Pettis;
		\item[$(b)$] Every continuous operator $T \colon X \longrightarrow X$ is Demi Weak Dunford-Pettis;
		\item[$(c)$] The identity operator of $X$ is Demi Weak Dunford–Pettis;
		\item[$(d)$] $X$ has the Dunford-Pettis property.
	\end{enumerate}
\end{theorem}

\begin{proof}
	"$(a) \Rightarrow (b)$" It is imediate from Proposition \ref{DwDP1}.
	
	"$(b) \Rightarrow (c)$" It is immediate
	
	"$(c) \Rightarrow (d)$" Let $(x_j)_{j=1}^{\infty}$ in $X$ and $(f_j)_{j=1}^{\infty}$ in $X'$ such that $x_j \overset{w}{\longrightarrow} 0$ and $f_j \overset{w}{\longrightarrow} 0$. Since 
	\begin{equation*}
		\|x_j - Id_X(x_j)\| = \|x_j-x_j\| = 0
	\end{equation*}
and the identity operator of $X$ is Weakly Demi Dunford–Pettis, then 
\begin{equation*}
	|f_j(x_j)|  \longrightarrow 0.
\end{equation*}

"$(d) \Rightarrow (a)$" Since $X$ has the Dunford-Pettis property, then each continuous operator $T \colon X \longrightarrow X$ is a Weakly Dunford-Pettis operator.
\end{proof}

It is important to note that a Banach space $X$ need not have the Dunford–Pettis property for there to exist a weak demi-Dunford–Pettis operator $T\colon X\to X$. For example, $T=-Id_{\ell_2}\colon\ell_2\to\ell_2$ is weak demi-Dunford–Pettis, although $\ell_2$ does not have the Dunford–Pettis property. 

\begin{lemma}{\label{SPDP}}
	Let $X$ be a Banach space such that $X'$ has the Schur property, then $X$ has Dunford-Pettis property
\end{lemma}



\begin{remark}
	It is important to note that every Banach space with the Schur property also has the Dunford–Pettis property. However, the converse implication does not hold in general. For instance, by Lemma \ref{SPDP}, the space $c_0$ has the Dunford–Pettis property, since its dual, $\ell_1$, has the Schur property (see \cite[Theorem $1.7$]{DJT95}). However, it is well known that $c_0$ does not have the Schur property.
\end{remark}

It is not difficult to see that every Demi Dunford–Pettis operator is, in particular, a Demi Weak Dunford–Pettis operator, but the converse need not be true. In fact, since $c_0$ has the Dunford–Pettis property, the Theorem \ref{WDDP1} implies that the identity operator $Id_{c_0}\colon c_0 \to c_0$ is Demi Weak Dunford–Pettis but not Demi Dunford–Pettis, since $c_0$ does not have the Schur property (see \cite[Theorem $2.4$]{BHJ22}). The next result provides a condition under which Weakly Demi Dunford–Pettis operators are Demi Dunford–Pettis.

\begin{theorem}\label{DDPWDDP}
	If $X$ is a reflexive Banach space, then every Demi Weak  Dunford–Pettis operator $T\colon X \longrightarrow X$ is a Dunford-Pettis operator.
\end{theorem}

\begin{proof}
	Let $X$ be a reflexive Banach space and let $T\colon X \longrightarrow X$ be a weakly Demi Dunford--Pettis operator. Suppose, by contradiction, that $T$ is not a Dunford-Pettis operator. Then there exist a sequence $(x_j)_{j=1}^{\infty}$ in $X$ and $\epsilon>0$ such that $x_j \xrightarrow{w} 0$,
\begin{equation*}
\|x_j-T(x_j)\|\to 0,
\end{equation*}
and
\begin{equation*}
\|x_j\|>\epsilon,\qquad \forall j\in\mathbb{N}.
\end{equation*}
By Hahn-Banach theorem, for each $j \in \mathbb{N}$  there exists $f_j \in X'$ such that
\begin{equation*}
	\|f_j\| = 1 \text{ and } f_j(x_j) = \|x_j\|.
\end{equation*}
By the Eberlein–Šmulian theorem, since $X'$ is reflexive and $(f_j)$ is bounded, there exist a subsequence $(f_{j_k})_{k=1}^{\infty}$ weak* convergent, that is, there exists a functional $f \in X'$ such that
\begin{equation*}
	f_{n_k}\overset{w^*}{\rightarrow} f
\end{equation*}
as $k \rightarrow \infty$. Since $X$ is a reflexive set then 
\begin{equation*}
	f_{n_k}\overset{w}{\rightarrow} f
\end{equation*}
as $k \rightarrow \infty$.
Since $T$ is a Demi weak Dunford-Pettis operator then by Proposition \ref{WDDPequi}
\begin{equation*}
	|f_{n_k}(x_{n_k})| \to 0,
\end{equation*}
which contradicts the fact that
\begin{equation*}
f_{j_k}(x_{j_k})=\|x_{j_k}\|>\epsilon
\end{equation*}
for every $k\in\mathbb{N}$.  Hence, $T$ is a Dunford-Pettis operator.
\end{proof}

Note that Theorem \ref{DDPWDDP} provides a condition under which Demi Weak Dunford-Pettis and Demi Dunford-Pettis linear operators coincide. By \cite[Proposition $2.7$]{BHJ22}, if $X$ is a reflexive Banach space, then the classes of Demicompact and Demi Dunford-Pettis linear operators also coincide. Recall that Petryshyn originally introduced the Demicompact operators in \cite{P66} as a generalization of nonlinear compact operators to address problems in fixed-point theory. Thus,  If $X$ is a reflexive Banach space, Theorem \ref{DDPWDDP} yields the following diagram:
\begin{equation*}
	\text{Demicompact } \overset{\cite[Proposition \ 2.7]{BHJ22}}{=} \text{ Demi Dunford-Pettis} \overset{\ref{DDPWDDP}}{=} \text{ Weakly Demi Dunford-Pettis}.
\end{equation*}

The following result asserts that an Dunford-Pettis perturbation of an Demi Weak Dunford-Pettis operator is again a Weakly Demi Dunford Pettis.

\begin{proposition}\label{DDP+DP}
	Let $T \colon X \rightarrow X$ be a Demi Weak Dunford-Pettis operator. If $S \colon X \rightarrow X$ is a Dunford-Pettis operator, then $T + S$ is a Demi weak Dunford-Pettis operator.
\end{proposition}

\begin{proof}
	Let $T \colon X \rightarrow X$ be a Demi Weak Dunford-Pettis operator and $S \colon X \rightarrow X$ be a Dunford-Pettis. Let $(x_j)_{j=1}^{\infty}$ a sequence in $E$ such that $x_j \xrightarrow{w} 0$ and $\|x_j - (T+S)(x_j)\| \rightarrow 0$. Consider also $(f_j)_{j=1}^{\infty}$ such that $f_j \overset{w}{\rightarrow} 0$. Since $S$ is a Dunford-Pettis operator, we have $\|S(x_j)\|\to 0$ as $j \to \infty$. Note that, for every $j \in \mathbb{N}$
	\begin{equation*}
		\|x_j - T(x_j)\| \le \|x_j - (T+S)(x_j)\| + \|S(x_j)\|,
	\end{equation*}
then
\begin{equation*}
\|x_j - T(x_j)\|\rightarrow 0,
\end{equation*}
as $j \rightarrow \infty$. Since $T \colon X \rightarrow X$ be a Demi Weak Dunford-Pettis operator then
\begin{equation*}
	|f_j(x_j)| \rightarrow 0
\end{equation*}
as $j \rightarrow \infty$. Hence, $T + S$ is a Demi  Weak Dunford-Pettis.
\end{proof}

In the next result, we establish a condition under which the matrix operator
\begin{equation*}
 \tilde{T} = \left(\begin{matrix}
		T_1 &T_2\\
		T_3 &T_4
	\end{matrix}\right)
\end{equation*}
is Demi Weak Dunford-Pettis. Here, $\widetilde{T}$ acts on the Banach space $X\times Y$, where $X$ and $Y$ are Banach spaces, and the operators $T_1\colon X\to X$, $T_2\colon Y\to X$, $T_3\colon X\to Y$ and $T_4\colon Y\to Y$ are bounded linear operators.
 
\begin{proposition}
Let $X$ and $Y$ be Banach spaces, where $Y$ has the Schur property. Let $T_1 \colon X \to X$, $T_4 \colon Y \to Y$, $T_2 \colon Y \to X$, and $T_3 \colon X \to Y$ be bounded linear operators. Consider the operator $\widetilde{T} \colon X \times Y \to X \times Y$ given by
\begin{equation*}
 \tilde{T} = \left(\begin{matrix}
		T_1 &T_2\\
		T_3 &T_4
	\end{matrix}\right)
\end{equation*}
If $T_1$ and $T_4$ are Demi Weak Dunford-Pettis operators and $T_2$ is a Dunford-Pettis operator, then $\widetilde{T}$ is a Demi weak Dunford-Pettis operator.
\end{proposition}

\begin{proof}
	Let $(z_j)_{j=1}^{\infty}$ be a sequence in $X \times Y$, where $z_j=(x_j,y_j)$, such that $z_j \overset{w}{\longrightarrow} 0$ as $j\to\infty$. Let $(h_j)_{j=1}^{\infty}$ be a sequence in $(X\times Y)'$, where $h_j=f_j+g_j$, such that $h_j\overset{w}{\longrightarrow}0$ as $j\to\infty$, with $f_j(x)=h_j(x,0)$ for every $x\in X$ and $g_j(y)=h_j(0,y)$ for every $y\in Y$. Suppose that,
\begin{equation*}
\|z_j-\widetilde{T}(z_j)\|_{X\times Y}\longrightarrow 0.
\end{equation*}
Since $z_j\overset{w}{\longrightarrow}0$ and $h_j\overset{w}{\longrightarrow}0$, it follows that
\begin{equation*}
x_j\overset{w}{\longrightarrow}0,\quad
y_j\overset{w}{\longrightarrow}0,\quad
f_j\overset{w}{\longrightarrow}0\quad \text{and }
g_j\overset{w}{\longrightarrow}0
\end{equation*}
as $j\to\infty$. Our goal is to prove that $|h_j(z_j)|\longrightarrow 0$. Since
\begin{equation*}
|h_j(z_j)|\leq |f_j(x_j)|+|g_j(y_j)|,
\end{equation*}
it suffices to show that $|f_j(x_j)|\longrightarrow 0$ and $|g_j(y_j)|\longrightarrow 0$ as $j\to\infty$.
	
	Note that
	\begin{align*}
		\|z_j - \tilde{T}(z_j)\|_{X\times Y} &= \|(x_j, y_j) - \tilde{T}(x_j, y_j)\|_{X \times Y}\\
		&= \|(x_j, y_j) - (T_1(x_j) + T_2(y_j), T_3(x_j) + T_4(y_j))\|_{X \times Y}\\
		&= \|(x_j - T_1(x_j) - T_2(y_j), y_j - T_3(x_j) - T_4(y_j))\|_{X \times Y}\\
		&= \|x_j - T_1(x_j) - T_2(y_j)\|_X + \|y_j - T_3(x_j) - T_4(y_j)\|_Y
	\end{align*}
and since $\|z_j - \tilde{T}(z_j)\|_{X\times Y} \rightarrow 0$ then $\|x_j - T_1(x_j) - T_2(y_j)\|_X \rightarrow 0$ and $\|y_j - T_3(x_j) - T_4(y_j)\|_Y \rightarrow 0$, as $j\rightarrow \infty$. On the one hand,
\begin{equation*}
	\|x_j - T_1(x_j)\|_X \le \|x_j - T_1(x_j) - T_2(y_j)\|_X + \|T_2(y_j)\|_X
\end{equation*}
and as $T_2$ is Dunford-Pettis operator, then
\begin{equation*}
	\|x_j - T_1(x_j)\|_X \rightarrow 0 \text{, as } j \rightarrow \infty.
\end{equation*} 
 Since $T_1$ is a Demi Weak Dunford-Pettis operator, we have $|f_j(x_j)| \rightarrow 0$, as $j \to \infty$. For all $\varphi \in Y'$
\begin{equation*}
|\varphi(y_j - T_4(y_j))| \le \|\varphi\| \|y_j - T_3(x_j) - T_4(y_j)\|_Y + |\varphi\circ T_3(x_j)|.
\end{equation*}
Then, $|\varphi(y_j - T_4(y_j))| \rightarrow 0$, for all $\varphi \in Y'$. It follows from the Schur property of $Y$ that
\begin{equation*}
	\|y_j - T_4(y_j)\|_Y \rightarrow 0, \text{ as } j \to \infty.
\end{equation*}
 Since $T_4$ is Weakly Demi Dunford-Pettis operator, we have $|g_j(y_j)| \rightarrow 0$, as $j \rightarrow \infty$.
\end{proof}

The next result provides a condition under which an operator is Demi Weak Dunford–Pettis whenever one of its powers has this property.


\begin{proposition}
Let $X$ be a Banach space, $T\colon X\to X$ an operator, and $m>1$ a natural number. If $T^m$ is Demi Weak Dunford-Pettis, then so is $T$.
\end{proposition}

\begin{proof}
	Let $\left((x_j)_{j=1}^{\infty}, (f_j)_{j=1}^{\infty} \right) \in X \times X'$ such that
	\begin{equation*}
		x_j \overset{w}{\rightarrow} 0 \text{ and } f_j \overset{w}{\rightarrow} 0
	\end{equation*}
and $\|x_j - T(x_j)\| \rightarrow 0$. Since,
\begin{equation*}
	\|T^{m-1}(x_j) - T^m(x_j)\| \le \|T^{m-1}\| \|x_j - T(x_j)\|, \forall j \in \mathbb{N} \text{ and } m > 1,
\end{equation*}
then $\|T^{m-1}(x_j) - T^m(x_j)\| \rightarrow 0$. Moreover, for every $j \in \mathbb{N}$
\begin{equation*}
	\|x_j - T^m(x_j)\| \le \|x_j - T(x_j)\| + \sum_{i=2}^{m}\|T^{i-1}(x_j) - T^i(x_j)\|.
\end{equation*}
Then, $\|x_j - T^m(x_j)\| \rightarrow 0$. hence, since $T^m$ is a Demi Weak Dunford-Pettis operator, we have $|f_j(x_j)| \rightarrow 0$, as $j \to \infty$.
\end{proof}

Following the same ideas as in the proof of the previous proposition, we obtain the following result.

\begin{proposition}
Let $X$ be a Banach space, $T\colon X\to X$ an operator, and $m>1$ an odd natural number. If $-T^m$ is Demi Weak Dunford-Pettis, then so is $-T$.
\end{proposition}


\begin{proposition}
	Let $X$ a Banach space and $T \colon X \rightarrow X$ be an operator. If $T$ is a Demi Weak Dunford-Pettis operator and $-T$ is a Demi Dunford-Pettis operator then $T^2$ is a Demi Weak Dunford-Pettis operator.
\end{proposition}

\begin{proof}
	Let $\left((x_j)_{j=1}^{\infty}, (f_j)_{j=1}^{\infty} \right) \in X \times X'$ such that
	\begin{equation*}
		x_j \overset{w}{\rightarrow} 0 \text{ and } f_j \overset{w}{\rightarrow} 0
	\end{equation*}
	and $\|x_j - T^2(x_j)\| \rightarrow 0$, as $j \rightarrow \infty$. Defining
	\begin{equation*}
		z_j = x_j - T(x_j)
	\end{equation*}
we have 
\begin{equation*}
	\|z_j + T(z_j)\| = \|x_j - T^2(x_j)\|
\end{equation*}
and so, 
\begin{equation*}
	\|z_j - (-T(z_j))\| = \|z_j + T(z_j)\| \rightarrow 0, \text{ as } j \rightarrow \infty.
\end{equation*}
Since $-T$ is a Demi Dunford-Pettis operator $$\|z_j\| \rightarrow 0, \text{ as } j \rightarrow \infty.$$ Note that
\begin{equation*}
	\|x_j - T(x_j)\| = \|z_j\| \rightarrow 0, \text{ as } j \rightarrow \infty.
\end{equation*}
Since $T$ is a Demi weak Dunford-Pettis operator, we have
\begin{equation*}
	|f_j(x_j)| \rightarrow 0,
\end{equation*}
as $j \rightarrow \infty$.

\end{proof}

Another result in this direction is the following.

\begin{proposition}
Let $X$ be a Banach space and $T\colon X\to X$ an operator. If $T$ is a Weak Dunford--Pettis operator and $-T$ is a Demi Weak Dunford--Pettis operator, then $T^2$ is a Demi Weak Dunford-Pettis operator.
\end{proposition}

\begin{proof}
	Let $\left((x_j)_{j=1}^{\infty}, (f_j)_{j=1}^{\infty} \right) \in X \times X'$ such that
		\begin{equation*}
			x_j \overset{w}{\rightarrow} 0 \text{ and } f_j \overset{w}{\rightarrow} 0
		\end{equation*}
		and $\|x_j - T^2(x_j)\| \rightarrow 0$, as $j \rightarrow \infty$. Defining
		\begin{equation*}
			z_j = x_j - T(x_j)
		\end{equation*}
		we have that
		\begin{equation*}
			\|z_j + T(z_j)\| = \|x_j - T^2(x_j)\|
		\end{equation*}
		and so, 
		\begin{equation*}
			\|z_j - (-T(z_j))\| = \|z_j + T(z_j)\| \rightarrow 0 
		\end{equation*}
		as $j \rightarrow \infty$. Since $-T$ is a Demi weak Dunford-Pettis operator $$|f_j(z_j)| \rightarrow 0, \text{ as } j \rightarrow \infty.$$ Note that,  for every $j \in \mathbb{N}$, we have
		\begin{equation*}
			|f_j(x_j)| \le |f_j(x_j) - f_j(T(x_j))| + |f_j(T(x_j))| = |f_j(z_j)| + |f_j(T(x_j))|.
		\end{equation*}
	Since $T$ is a weak Dunford-Pettis operator operator, then
	\begin{equation*}
		|f_j(x_j)| \rightarrow 0,
	\end{equation*}
as $j \rightarrow \infty$. Hence, $T^2$ is a Demi Weak Dunford-Pettis operator.
\end{proof}

\section{Demi Weak Dunford-Pettis Operators on Banach Lattices}


In this section, we present some results on Demi Weak Dunford-Pettis operators in the setting of Banach lattices. The next two results address the problem of domination of Demi Weak Dunford-Pettis operators by central operators. Recall the definition of a central operator (see \cite[Definition $3.28$]{AA02}): an operator $T\colon X\to X$ is said to be central if there exists $\lambda>0$ such that $|T(x)|\leq \lambda|x|$ for every $x\in X$.

\begin{theorem}\label{Desig}
	Let $X$ be a Banach Lattice and $T, S\colon X \rightarrow X$ positive operators satisfying $0 \le S \le T \le Id_X$. If $T$ is a Demi Weak Dunford-Pettis operator, then so is $S$.
\end{theorem}

\begin{proof}
	Let $T, S\colon X \rightarrow X$ be positive operators satisfying 
\begin{equation*}
0 \le S \le T \le Id_X.
\end{equation*}
 Let $\left((x_j)_{j=1}^{\infty}, (f_j)_{j=1}^{\infty} \right) \in X \times X'$ such that
		\begin{equation*}
			x_j \overset{w}{\rightarrow} 0 \text{ and } f_j \overset{w}{\rightarrow} 0
		\end{equation*}
		and $\|x_j - S(x_j)\| \rightarrow 0$, as $j \rightarrow \infty$. We shall prove that
	\begin{equation*}
		|f_j(x_j)| \rightarrow 0, \text{ as } j \to \infty.
	\end{equation*}
Follows from \cite[Theorem $3.30$]{AA02} that
\begin{equation*}
	|(Id_X-T)(x_j)| = |Id_X-T|(|x_j|) = (Id_X-T)(|x_j|).
\end{equation*}
Since $0 \le S \le T \le Id_X$ 
\begin{equation*}
	|(Id_X-T)(x_j)| = (Id_X-T)(|x_j|) \le (Id_X-S)(|x_j|) = |(Id_X-S)(x_j)|.
\end{equation*}
So,
\begin{equation*}
	\|x_j - T(x_j)\|  \le \|(Id_X-S)(x_j)\|
\end{equation*}
for every $j \in \mathbb{N}$, and thus,
\begin{equation*}
	\|x_j - T(x_j)\| \rightarrow 0,
\end{equation*}
as $j \rightarrow \infty$. Since, $T$ is a Demi Weak Dunford-Pettis operator then
\begin{equation*}
	|f_j(x_j)| \rightarrow 0,
\end{equation*}
as $j \rightarrow \infty$. Hence, $S$ is a Demi Weak Dunford-Pettis operator.
\end{proof}

\begin{remark}
	Let $R, S, T$ operators on a Banach lattice $X$ such that $R \le S \le T \le Id_X + R$
	holds. If $R$ and $T$ are  Demi Weak Dunford-Pettis, the operator $S$ is or no a Demi Weak Dunford-Pettis. In fact, this problem is equivalent to $0 \le S - R \le T - R \le Id_X$ where the operators $0$ and $T - R$ are  Demi Weak Dunford-Pettis. But, if $R$ and $T$ are   Demi Weak Dunford-Pettis operators, $T - R$ is not necessary a Demi Weak Dunford-Pettis operator.
\end{remark}

\begin{corollary}\label{Desig2}
	Let $T, R, S$ be operadors on Banach Lattices $X$ such that $R \le S \le T \le Id_X + R$ holds. If $R$ is a Dunford-Pettis operator and $T$ is a Demi Weak Dunford-Pettis operator then $S$ is a Demi Weak Dunford-Pettis operator.
\end{corollary}

\begin{proof}
	Let $T, R, S$ be operadors on Banach Lattices $X$ satisfying $R \le S \le T \le Id_X + R$ holds. It is easy to see that the inequalities $R \le S \le T \le Id_X+R$ are equivalent to $0 \le S-R \le T-R \le Id_X$. Since $R$ is a Dunford-Pettis operator, Proposition \ref{DDP+DP} implies that $T-R$ is a Demi Weak Dunford-Pettis operator. Hence, by Theorem \ref{Desig}, $S-R$ is also a Demi Weak Dunford–Pettis operator.
	
	Again, by Proposition \ref{DDP+DP}, since $S-R$ is a Demi Weak Dunford--Pettis operator, it follows that
	\begin{equation*}
		S = (S-R) + R,
	\end{equation*}
is a  Demi Weak Dunford-Pettis operator.
\end{proof}

For the next result, two important concepts will be needed: disjointness-preserving operators and lattice homomorphisms.

\begin{definition}
	\begin{itemize}
		\item[$(a)$] An operator $T\colon X \rightarrow Y$ betwen two Riesz space is said to preserve disjointness whenever $x\perp y$ in $X$ implies $T(x)\perp T(y)$ in $Y$.
		\item[$(b)$] An operator $T\colon X \rightarrow Y$ betwen two Riesz space is said to be a lattice homomophism whenever $T(x \vee y) = T(x) \vee T(y)$ holds for all $x, y \in X$.
	\end{itemize}
\end{definition}

\begin{remark}
	It is immediate that every lattice homomorphism is, in particular, a positive operator, since, for every $x \in X^{+}$
	\begin{equation*}
		T(x) = T(x \vee 0) = T(x) \vee T(0) = T(x) \vee 0 = T(x)^{+} \ge 0.
	\end{equation*}
\end{remark}

An important consequence of \cite[Theorem $2.14$]{AB85} is that every lattice homomorphism $T\colon X\to Y$ is, in particular, disjointness preserving and satisfies
\begin{equation*}
	|T(x)| = T(|x|)
\end{equation*}
for all $x \in X$.

\begin{theorem}
	Let $X$ be a Banach lattice and $S,T\colon X \rightarrow X$ be positive operators. If $Id_{X} \le S \le T$ such that $T$ is a lattice homomorphism and $S$ is a Demi Weak Dunford-Pettis operator, then $T$ is a Demi Weak Dunford-Pettis operator.
\end{theorem}

\begin{proof}
	Let $((x_j)_{j=1}^{\infty}, (f_j)_{j=1}^{\infty})$ in $X\times X'$ weakly null sequences with $\|x_j - T(x_j)\| \rightarrow 0$, as $j \rightarrow \infty$. We shall prove that $$|f_j(x_j)| \rightarrow 0, \text{ as } j \rightarrow \infty.$$ Since $T$ is a lattice homomorphism, we have
	\begin{equation*}
		|(T-Id_{X})(x_j)| = |T(x_j)-x_j| \ge |T(x_j)| - |x_j| = T(|x_j|) - |x_j| = (T-Id_{X})(|x_j|)
	\end{equation*}
for all $j \in \mathbb{N}$. Thus,
\begin{equation*}
	|x_j - S(x_j)|	= |(Id_{X} - S)(x_j)| \le (S-Id_{X})(|x_j|) \le (T-Id_{X})(|x_j|) \le |(T-Id_{X})(x_j)| 
\end{equation*}
for all $j \in \mathbb{N}$. So,
\begin{equation*}
	\|x_j - S(x_j)\| \le \|(T-Id_{X})(x_j)\|
\end{equation*}
for all $j \in \mathbb{N}$. Therefore,
\begin{equation*}
	\|x_j - S(x_j)\| \rightarrow 0,
\end{equation*}
as $j \rightarrow \infty$. Since, $S$ is a Demi Weak Dunford-Pettis operator
\begin{equation*}
	|f_j(x_j)| \rightarrow 0,
\end{equation*}
as $j \rightarrow \infty$.
\end{proof}

\end{document}